\documentclass[11pt, a4paper]{article}
\usepackage{graphicx}
\usepackage{tikz}
\usepackage{amssymb,amsmath,amsthm}
\usepackage{verbatim}
\usepackage{color}
\usepackage[shortlabels]{enumitem}
\usepackage{dsfont}
\usepackage[margin=2.75cm]{geometry}
\usepackage{thmtools, thm-restate}
\usepackage{hyperref}
\usepackage[capitalise]{cleveref}
\newcommand{\F}{\mathbb{F}}
\newcommand{\Z}{\mathbb{Z}}
\newcommand{\Q}{\mathbb{Q}}
\newcommand{\N}{\mathbb{N}}
\newcommand{\PP}{\mathbb{P}}
\newcommand{\ZZ}{\mathbb{Z}}
\newcommand{\Zpk}{\mathcal{R}}
\newcommand{\PZpk}{\mathbb{P}\mathcal{R}^{n-1}}

\DeclareMathOperator{\len}{length}

\declaretheorem{theorem}
\declaretheorem[sibling=theorem]{corollary, lemma, proposition, question, definition, conjecture, example, remark}

\title{Covering points with planes}
\author{Hailong Dao \thanks{Department of Mathematics, University of Kansas, Lawrence, KS 66045, USA} \and Manik Dhar \thanks{Department of Mathematics, Massachusetts Institute of Technology, MA, USA} \and Izabella Łaba \thanks{Department of Mathematics, University of British Columbia, Vancouver,
B.C. V6T 1Z2, Canada}
\and Ben Lund \thanks{
Discrete Mathematics Group (DIMAG), Institute for Basic Science (IBS), Daejeon, South Korea.}
}
\date{January 2025}

\begin{document}

\maketitle

\begin{abstract}
    Suppose that each proper subset of a set $S$ of points in a vector space is contained in the union of planes of specified dimensions, but $S$ itself is not contained in any such union.
    How large can $|S|$ be?

    We prove a general upper bound on $|S|$, which is tight in some cases, for example when all of the planes have the same dimension.
    We produce an example showing that this upper bound does not hold for point sets whose proper subsets are covered by lines in $(\mathbb{Z}/p^k\mathbb{Z})^2$ with $k\geq 2$, and prove an upper bound in this case.
    We also investigate the analogous problem for general matroids.
    
\end{abstract}

\section{Introduction}

Let $S \subset \F^n$ be a finite set of points in an $n$-dimensional vector space over a field $\F$. Our project stems from the following natural question:
\begin{question}
    Let $(D)$ be some fixed degenerate condition (for instance, $(D)$ can be ``lying on an union of $t$ hyperplanes"). If each proper subset of $S$ satisfies $(D)$, must $S$ also be $(D)$? 
\end{question}
 Clearly, such statement would be helpful in  various contexts, especially if one needs to apply some sort of induction. Equally clearly, in order for the statement to hold, $S$ must be large enough: each proper subset of $3$ points lies on a line, but the $3$ points might not!  

One of our main results settles the above question for a large collection of degenerate conditions. Let $V$ be a finite set of dimension vectors, say $V=\{(2,1),(1,1,1)\}$. Let $(D_V)$ be the condition that $S$ lies on a plane arrangement with dimension vectors from $V$ (so, in this example, $(D_V)$ means ``the points lie on either an union of a plane and a line, or three lines"). We prove:

\begin{theorem}\label{th:generalVBound}
  For any $V$, there is a constant $C(V)$ such that if any subset of $S$ of size at most $C(V)$ satisfies $(D_V)$, then $S$ is $(D_V)$.   
\end{theorem}


At this point, a reader might reasonably ask why we would want to study the above problem with a set of multiple distinct dimension vectors. To provide motivation we list below a number of situations where understanding such conditions is desirable:

\begin{itemize}
    \item Algebraic geometry: Let $X$ be a finite set of points in the complex projective space $\PP^n$. A well-known result (\cite[Proposition 1.2 and 1.5]{Quadrics} or \cite[Theorem 8.18]{eisenbud2006geometry}) says that the Betti number $\beta_{n,n+1}$ of the coordinate ring of $X$ is non-zero if and only if $X$ lies on the union of two planes whose sum of dimension is less than $n$ (in other words, $X$ satisfies condition $D_V$ for $V=\{(a,b), a+b<n\}$). For similar statements and some fascinating open questions, see \cite{AST_1993__218__187_0, eisenbud1992finite, eisenbud1989remarks, green1988some}.
    \item Matroid theory: A matroid $M$ is $2$-connected \cite[Ch. 4]{oxley2011matroid} if and only if the elements of $M$ are not contained in the union of any two flats whose sum of ranks is at most the rank of $M$ (this generalizes the condition of the previous item).
    Murty \cite{murty1974extremal} showed that a rank $r$ matroid that is minimally $2$-connected has at most $2(r-1)$ elements, and Oxley \cite{oxley1981connectivity} gave a complete characterization of mimimally $2$-connected matroids with $2(r-1)$ elements.
    These results have many applications, and several related questions remain open.
    For example, see \cref{sec:combinatorial} for a discussion of $t$-thick matroids. 
    \item Combinatorial geometry: a set of points is $k$-degenerate if it is contained in a set $L_1,\ldots,L_t$ of planes, each with dimension at least $1$, whose dimensions add up to at most $k$.
    Point sets that span few $k$-planes \cite{lund2018essential,do2020extending} or whose average $k$-planes have many points \cite{campbell2023average,campbell2024characterizing} are almost $k$-degenerate.
    
\item Intrinsic interest: We found the problem itself to be interesting even for very simple set of dimension vectors, such as those with only $1$s and $0$s. See \cref{th:withZeros}. It seems any definitive result would require a combination of techniques, and we hope this work will lead to further investigations in this direction. 
    
\end{itemize}

For a specific condition $(D_V)$, it is of considerable interest to find the best possible bound for $C(V)$. This is in general a difficult and quite fascinating problem. One important case in which we obtain a complete solution is when $V=(n-1,...,n-1)$ with $t$ entries; in other words, $(D_V)$ means  ``lying on a union of $t$ hyperplanes":
\bigskip

\begin{theorem}\label{th:introHyperplaneCover}
    If $S \subseteq \F^n$ and every subset $T \subseteq S$ with $|T| \leq \binom{n+t}{n}$ is contained in the union of a set of $t$ affine hyperplanes, then $S$ is contained in the union of a set of $t$ affine hyperplanes.
\end{theorem}

The following example shows that the function $\binom{n+t}{n}$ in \cref{th:introHyperplaneCover} cannot be replaced by anything smaller.

\begin{example}\label{ex:triangle}
Let $T_t$ be the set of integer points in the convex hull of $(0,0), (0,t), (t,0)$. Then $T_t$ is not in the union of any set of $T_t$ lines, but each proper subset of $T_t$ is contained in the union of some set of $t$ lines.

We first show that $T_t$ is not contained in the union of any set of $t$ lines.
Suppose, toward a contradiction, that there is a least integer $t_0$ such that $T=T_{t_0}$ is covered by $t_0$ lines, and let $\mathcal{L}$ be a set of $t_0$ lines that covers $T$.
Let $L$ be the line defined by $x=0$.
Since $|L \cap T| = t+1$, we must have $L \in \mathcal{L}$.
Hence, $T \setminus \{L\}$ is covered by $t_0-1$ lines, which contradicts the minimality of $t_0$.

Now we show that, for each $P = (p_1,p_2) \in T_t$, the set $T_t \setminus \{P\}$ is contained in the union of some set of $t$ lines.
Indeed, choose the lines defined by the equations
\[\{x=c:c \in \mathbb{Z} \cap [0,p_1)\} \cup \{y=c:c \in \mathbb{Z} \cap [0,p_2)\} \cup \{x+y=c: c \in \mathbb{Z} \cap (p_2+p_1,t]\}. \]
The total number of lines chosen is $p_1 + p_2 + (t-p_1-p_2) = t$.
For any point $(a,b) \in T_t \setminus \{P\}$, either $a < p_1$, or $b < p_2$, or $a + b > p_1 + p_2$, and hence $(a,b)$ is contained in the union of the lines.

In higher dimensions, let $T_{n,t} \subset \mathbb{R}^n$ be the set of integer points satisfying $x_i \geq 0$ for each $i \in [n]$, and $x_1+x_2+ \ldots + x_n \leq t$.
An argument analogous to that given above shows that $T_{n,t}$ is not contained in the union of $t$ hyperplanes, and each proper subset of $T_{n,t}$ is.
\end{example}

    We also study this problem over more general rings. Recently, there have been some works on combinatorial geometry over general rings, and this setting makes classical questions quite interesting and subtle (see \cite{Dhar2024,laba2024}).  In \Cref{sec:generalAlgebraic} we are able to extend \Cref{th:generalVBound} to Artinian rings and the rest of the paper focuses on the ring $\Z/p^k\Z$. The earlier explicit bounds over fields do not apply because two lines at a small `angle' can intersect on a large number of points. As $k$ increases, the number of `scales' in the ring increase. In this setting we often want to study problems with increasing $k$ as it corresponds to thin tubes and balls over the $p$-adics. We  restrict our focus to the case of $n=2$.

We are able to show the following upper bound.

\begin{theorem}\label{th:introP-adic}
    If $S \subseteq (\Z/p^k\Z)^2$ and every subset $T \subseteq S$ with $$|T| \leq t(1+k^{-1}t)^k$$ is contained in the union of $t$ lines, then $S$ is contained in the union of $t$ lines. 
    When $k>t$ and $t<p$, we have the same result for $|T|\le t2^t$.
\end{theorem}

The bound follows from an inductive counting argument. 
In the converse direction, in Section \ref{sec:fullExample} we construct nearly-covered sets in $(\Z/p^k\Z)^2$ with $k\geq 2$ that are larger than would be possible in $(\Z/p\Z)^2$. The construction relies on the availability of multiple scales. One special case is as follows.


\begin{theorem}\label{th:p-adicConstruction}
Assume that $p>k\geq 3$ and 
$2\leq t<\frac{1}{4}\sqrt{p}$. Let $t':=t+\lfloor (k-1)/2\rfloor-1$. Then there exists a set $S$ in $(\Z/p^k\Z)^2$ of size  
\begin{equation}\label{eq:S-specialcase}
|S|=  2^{\left\lfloor \frac{k-1}{2}\right\rfloor} 
\left( \binom{t+1}{2}+1\right) -1
\end{equation}
such that $S$ cannot be covered by $t'$ lines but $S\setminus \{x\}$ for all $x\in S$ can be covered by $t'$ lines.
\end{theorem}

In particular, if we fix the parameter $t$ in Theorem \ref{th:p-adicConstruction} and then allow both $p$ and $k$ to be large, the size of the set $S$ in the theorem is bounded from below by $2^{t'-t}$. This shows that the second upper bound in Theorem \ref{th:introP-adic} is close to optimal at least in some cases.

\subsection{General setup}

A {\em dimension vector} is a vector $\vec{v} = (v_1, \ldots,v_t)$ over the non-negative integers with $v_1 \geq \ldots \geq v_t$.
A set $S$ of points is {\em covered} by a set $\{L_1, \ldots, L_t\}$ of affine planes if it is contained in their union.
$S$ is covered by a dimension vector $\vec{v} = (v_1,\ldots, v_t)$ if it is covered by a set $\{L_1, \ldots, L_t\}$ of planes with $\dim(L_i) = v_i$ for each $i \in [t]$.
$S$ is covered by a set $V$ of dimension vectors if it is covered by some $\vec{v} \in V$.
$S$ is {\em nearly covered} by a set $V$ of dimension vectors if it is not covered by any $\vec{v} \in V$, but every proper subset $R \subset S$ is covered by some $\vec{v}_R \in V$.

For a finite set $V$ of dimension vectors and affine space $A$ over a field, denote by $C_A(V)$ the least $N$ such that, for every finite $S \subset A$, if every subset of $S$ of size at most $N$ is covered by $V$, then $S$ is covered by $V$.
Equivalently, $C_A(V)$ is equal to the largest number of points in any set that is nearly covered by $V$.

\subsection{Organization}
The proof of \cref{th:generalVBound,th:introHyperplaneCover} is in \cref{sec:fieldAlgebraic}.
\Cref{sec:generalAlgebraic} generalizes the algebraic proof given in \cref{sec:fieldAlgebraic} to the setting of Artinian rings.
In \cref{sec:combinatorial}, we discuss a generalization to the setting of matroids.
We consider the $p$-adic variant of our question in Section \ref{sec:padic-full}. \Cref{sec:p-adic} introduces the definitions and basic results on the geometry of modules over $\mathbb{Z}/p^k\mathbb{Z}$.
In \Cref{sec:fullExample}, we demonstrate a simple example over $(\mathbb{Z}/p^k\mathbb{Z})^2$ that has more points than that described in \cref{ex:triangle}, and then iterate it 
to obtain \cref{th:p-adicConstruction}.
\Cref{sec:p-adicBound} contains the proof of \cref{th:introP-adic}.

\section{Algebraic approach}\label{sec:fieldAlgebraic}

Fix a field $\F$ and consider a set $S$ of points in $\F^n$. We view $R=\F[x_1,\dots, x_n]$ as a $\F$-vector space. Let $R_d$, ($R_{\leq d}$) denote the subspaces of polynomials of degree $d$ (at most $d$).  We say that a subspace $Y$ of $R$ covers $S$ if $Y\subset I(S)$, where $I(S)$ denotes the ideal of polynomials vanishing on $S$. 

The following is trivial. 
\begin{lemma}\label{lem1field}
\begin{enumerate}
\item If $Y$ covers $S$, then so does any subspace of $Y$.
\item     If $Y_i$ covers $S_i$ for each $i$, then $\sum Y_i$ covers $\cap S_i$.
\end{enumerate}

\end{lemma}

\begin{proposition}\label{prop1}
    Let $S=\{P_1,\dots, P_r\}$ be a set of points in $\F^n$. Suppose that $S\setminus \{P_i\}$ is covered by $Y_i$. If 
    $$ r> \dim \sum_{i=1}^r Y_i - \dim Y_1+1,$$
    then one of $Y_i$ covers $S$. 
\end{proposition}

\begin{proof}
Consider the chain of subspaces $Y_1\subseteq  Y_1+Y_2 \subseteq \dots \subseteq Y_1+\dots+ Y_r$. The condition $r> \dim \sum_{i=1}^t Y_i - \dim Y_1+1$ forces at least  one  equality in the chain, i.e., $Y_{i+1}\subseteq Y_1+\dots Y_i$ for some $i$. By \cref{lem1field},   $Y_{i+1}$ covers $\bigcap_{j=1}^i\{S\setminus \{P_i\}\}$ which contains $P_{i+1}$, thus $Y_{i+1}$ covers the whole $S$. 
\end{proof}

\begin{theorem}\label{th:fieldUpperBound}
    If $V$ is a set of dimension vectors, each with at most $t$ coordinates, and $k = \max_{\vec{v} \in V} \max_{i} v_i$,
    then
    \[C_A(V) \leq \binom{t + k + 1}{k+1}.\]
\end{theorem}
\begin{proof}
    Let $S'=\{P_1', \ldots, P_r'\}$ be a set of points in an affine $A$ space over a field $\F'$.
    For each $i \in [r]$, suppose that $A_i'$ is an $V$-covering set for $S \setminus \{P_i\}$.

    Let $\pi$ be a generic projection from $A$ to $\F^{k+1}$, where $\F$ is possibly an extension of $\F'$.
    More precisely, $\pi$ is projection from a point that is not contained in any plane spanned by the points of $S'$.
    For $i \in [r]$, let $P_i = \pi(P_i')$ and let $A_i = \pi(A_i')$.
    Let $S = \{P_1, \ldots, P_r\}$.
    Note that $A_i$ is a $V$-covering set for $S \setminus \{P_i\}$.

    Fix an $i$, and let $\Gamma_j$ be the vanishing ideal of the affine plane $L_j \in A_i$.
    Each $\Gamma_j$ is generated by linear forms.
    The product $\Gamma$ of those ideals $\Gamma_j$ will be in the ideal $I(S \setminus \{P_i\})$.
    Each generator of $\Gamma$, being products of at most $t$ linear forms, lives in $R_{\leq t}$.
    Thus, the span of those generators of $\Gamma$ lives in $R_{\leq t}$ and covers $S \setminus \{P_i\}$.
    Since $\dim R_{\leq t} = \binom{k+1+t}{k+1}$, the conclusion of the theorem follows directly from \cref{prop1}.
\end{proof}

\Cref{ex:triangle} shows that \cref{th:fieldUpperBound} is tight in the case that $V = \{(k,k,\ldots,k)\}$, but it is not tight in general.
In the remainder of this section, we show that \Cref{prop1} can also be used to prove tight bounds for certain dimension vectors with all entries either $k$ or zero.
We first give the construction, which generalizes the construction given in \cref{ex:triangle}.

\begin{example}\label{ex:missingTriangle}
    Let $0 \leq s \leq t$, and let $t_1 = t - s$ and $t_2 = \binom{s + n}{n} - 1$.
    Let $T = T_{n,t} \subset \mathbb{R}^n$ be the set of integer points such that $x_i \geq 0$ for $i \in [n]$ and $x_1 + x_2 + \ldots + x_n \leq t$.
    Note that $|T| = \binom{t+n}{n}$.
    We will show that $T$ is not contained in the union of any collection of $t_1$ hyperplanes and $t_2$ points, but that every proper subset of $T$ is contained in such a union.
    This implies that, if $\vec{v}$ is the dimension vector with $t$ entries, of which $t_1$ are equal to $n-1$ and $t_2$ are zeros, then $C(\vec{v}) \geq \binom{t + n}{n}$.

    First, we show that $T$ is not contained in the union of $t_1$ hyperplanes and $t_2$ points.
    We proceed by induction on $n$, $t_1$, and $t_2$.
    In the case $n=1$, hyperplanes are just points so the claim is that $|T|=t+1 > t_1 + t_2 = t_1+ \binom{s+1}{1}-1 = t$ points, which is true.
    In the case $t_1 = 0$, the claim is that $|T| = \binom{s+n}{n} > t_2 = \binom{s+n}{n} - 1$, which is true.
    The case $t_2 = 0$ is exactly \cref{ex:triangle}.
    
    Suppose, toward a contradiction, that there is a set $\mathcal{L}$ with $t_1$ hyperplanes and $t_2$ points that covers $T$.
    Let $H_0$ be the hyperplane defined by $x_1 = 0$.
    Note that $T \cap H_0$ is a copy of $T_{n-1,t}$, and $T \setminus H_0$ is a copy of $T_{n,t-1}$.
    If $H_0 \in \mathcal{L}$, then $T \setminus H_0$ is covered by $t_1 - 1$ hyperplanes and $t_2$ points, which contradicts the inductive hypothesis on $t_1$.
    Otherwise, $H_0 \cap T$ is covered by $t_1$ hyperplanes and $t_2' \leq t_2$ points.
    By induction on $n$, we have that $t_2' \geq \binom{s+n-1}{n-1}$.
    Hence, $T \setminus H_0$ is covered by $t_1$ hyperplanes and $t_2 - t_2' \leq \binom{s+n}{n} - \binom{s+n-1}{n-1} - 1 \leq \binom{s+n-1}{n} - 1$ points.
    This contradicts the inductive hypothesis on $t_2$.
    
    Now we show that each proper subset of $T$ is contained in the union of $t_1$ hyperplanes and $t_2$ points.
    Let $P \in T$ be an arbitrary point.
    For any integer point $A = (a_1, a_2, \ldots, a_n)$ such that $a_1+\ldots+a_n+s\leq t$, let $S_A$ be the set of integer points such that $x_i \geq a_i$ for $i \in [n]$ and $x_1+x_2+\ldots +x_n \leq a_1 + a_2 + \ldots +a_n + s$.
    Note that $|S| = \binom{s+n}{n} = t_2+1$.
    Choose $A$ so that $P \in S_A \subseteq T$ - see \cref{fig:missingTriangle} for an illustration.
    It is easy to check that $T \setminus \{P\}$ is covered by the $t_2$ points of $S \setminus \{P\}$ together with the hyperplanes defined by the equations
    \[\{x_i = c: i \in [n], c \in \mathbb{Z} \cap [0,a_1)\} \cup \{x_1+x_2+\ldots+x_n = c: c \in \mathbb{Z} \cap (a_1 + a_2 + \ldots + s,t]\}.\]
    The total number of such hyperplanes is $a_1+a_2+\ldots+a_n + (t-a_1-\ldots-a_n - s) = t_1$.
\end{example}

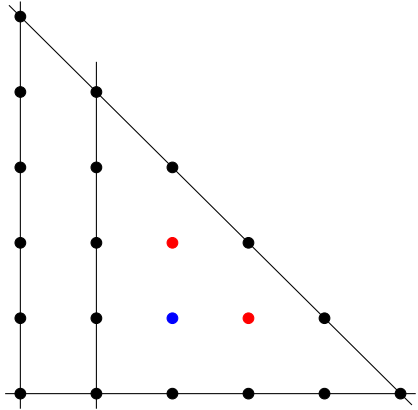
\begin{figure}[h!]
\begin{center}
\begin{tikzpicture}[scale=1]

\filldraw[black] (0,0) circle (2pt);
\filldraw[black] (1,0) circle (2pt);
\filldraw[black] (2,0) circle (2pt);
\filldraw[black] (3,0) circle (2pt);
\filldraw[black] (4,0) circle (2pt);
\filldraw[black] (5,0) circle (2pt);
\filldraw[black] (0,1) circle (2pt);
\filldraw[black] (1,1) circle (2pt);
\filldraw[blue] (2,1) circle (2pt);
\filldraw[red] (3,1) circle (2pt);
\filldraw[black] (4,1) circle (2pt);
\filldraw[black] (0,2) circle (2pt);
\filldraw[black] (1,2) circle (2pt);
\filldraw[red] (2,2) circle (2pt);
\filldraw[black] (3,2) circle (2pt);
\filldraw[black] (0,3) circle (2pt);
\filldraw[black] (1,3) circle (2pt);
\filldraw[black] (2,3) circle (2pt);
\filldraw[black] (0,4) circle (2pt);
\filldraw[black] (1,4) circle (2pt);
\filldraw[black] (0,5) circle (2pt);

\draw[black] (0,-.2)--(0,5.2);
\draw[black] (1,-.2)--(1,4.4);
\draw[black] (-.2,0)--(5.2,0);
\draw[black] (5.15,-.15)--(-.15,5.15);

\end{tikzpicture}
\end{center}
\caption{Possible set of points and lines described in \Cref{ex:missingTriangle} with $t=5$ and $s=1$. The point $P$ is blue, the additional points of $S$ are red, and the points of $T \setminus S$ are black. Four lines and two points suffice to cover $T \setminus \{P\}$.} \label{fig:missingTriangle}
\end{figure}

Here comes the proof that shows that this example is as large as possible.

\begin{theorem}\label{th:withZeros}
Let $n \geq 2$, let $0 \leq s \leq t$, let $t_1 = t - s$, and let $t_2 = \binom{s +n}{n}-1$.
Let $\vec{v}$ be a dimension vector with $t_1$ entries $n-1$ and $t_2$ zeros.
Then, $C(\vec{v}) = \binom{t_1 + r + n}{n}$.
\end{theorem}
\begin{proof}
    Let $X = \{P_1,\ldots,P_n\}$ be a set of points such that $X - \{P_i\}$ is covered by a set $\mathcal{L}_i$ of $t_1$ lines and $t_2$ points, but $X$ itself is not covered by any such set.
    Since any set of $t_2 = \binom{s+n}{n}-1$ points is contained in the zero set of a polynomial of degree $s$, we have that $X - \{P_i\}$ is contained in the zero set of a polynomial of degree $t_1 + s$.
    The rest of the argument follows as in the proof of \cref{th:fieldUpperBound}.
\end{proof}

\section{Generalized algebraic approach}\label{sec:generalAlgebraic}

In this section we prove a generalized version of \Cref{th:generalVBound} using ring-theoretic methods. While it gives the existence of a bound over Artinian rings, effective applications are more limited than the field case: we need estimates on degree of generators of vanishing ideals, which in this situation is much more subtle. Even the concept of ``lines" and ``planes" over rings such as $\Z/p^k\Z$ requires more care and alternative approaches, see Section \ref{sec:padic-full}.

Fix an Artinian ring $\F$ and consider a set $X$ of points in $\F^n$. We view $R=\F[x_1,\dots, x_n]$ as a free $\F$-module. Let $R_d$, ($R_{\leq d}$) denote the free, finitely generated submodule of polynomials of degree $d$ (at most $d$).  We say that an $\F$-submodule $Y$ of $R$ covers $X$ if $Y\subset I(X)$, where $I(X)$ denotes the ideal of polynomials vanishing on $X$. 

The following is trivial. 
\begin{lemma}\label{lem1}
\begin{enumerate}
\item If $Y$ covers $X$, then so does any submodule of $Y$.
\item     If $Y_i$ covers $X_i$ for each $i$ inside a finite set $S$, then $\sum Y_i$ covers $\cap X_i$.
\end{enumerate}    
\end{lemma}

Recall the length of a finitely generated $\F$-module $M$, denoted by $\len_{F}(M)$, is the supremum of lengths of all chains of submodules in $M$.  As $F$ is Artinian, this is finite. If $M$ is free of rank $r$, $M\cong \F^r$, then $\len(M)=r\len(R)$. For reference on basic facts about length,  see \cite[\href{https://stacks.math.columbia.edu/tag/00IU}{Tag 00IU}]{stacks-project}.

\begin{proposition}\label{prop2.1}
    Let $X=\{P_1,\dots, P_r\}$ be a set of points in $\F^n$. Suppose that $X-\{P_i\}$ is covered by $Y_i$. If 
    $$ r> \len(\sum_{i=1}^r Y_i) - \len Y_1+1,$$
    then one of $Y_i$ covers $X$. 
\end{proposition}

\begin{proof}
Consider the chain of submodule $Y_1\subseteq  Y_1+Y_2 \subseteq \dots \subseteq Y_1+\dots+ Y_r$. The condition $r> \len(\sum_{i=1}^t Y_i) - \dim Y_1+1$ forces at least  one  equality in the chain, i.e., $Y_{i+1}\subseteq Y_1+\dots Y_i$ for some $i$. By \cref{lem1},   $Y_{i+1}$ covers $\cap_{j=1}^i\{X-\{P_i\}\}$ which contains $P_{i+1}$, thus $Y_{i+1}$ covers the whole $X$. 
\end{proof}

To state the next Theorem, we need one more definition. 

\begin{definition}
For an ideal $I\subset R$, we let $g(I)$ be the infimum (over all finite system of generators $G$  of $I$) of the maximal degree of the polynomials in $G$. 
\end{definition}

\begin{theorem}\label{thmF}
Let $X=\{P_1,\dots, P_r\}$ be a set of points in $\F^n$. Suppose that $I_i\subset R$ covers $X-\{P_i\}$. If $d$ is the maximal degree of $g(I_i)$, and $r\geq (\len(\F))\binom{n+d}{d}+1$, then at least one $I_i$ covers $X$. 
\end{theorem}

\begin{proof}
    Fix an $i$. By definition, $I_i$ can be generated by a system of polynomials $G_i$ in $R_{\leq d}$. Let $Y_i$ denote the $\F$-submodule generated by $G_i$. Each $Y_i$ covers $X-\{P_i\}$, and $\sum Y_i \subset R_{\leq d}$. Finally, note that $R_{\leq d}$ is a free $\F$-module of rank $\binom{n+d}{d}$, thus its length is $\len(\F)\binom{n+d}{d}$. We now apply \cref{prop2.1} to finish the proof. 
\end{proof}

\begin{example}
We give some basic examples of length to apply the bound in \cref{thmF}.  We have $\len(\F)=1$ if and only if $\F$ is a field. If $\F=\Z/n\Z$ with $n=\prod p_i^{n_i}$ with distict primes $p_i$ and $n_i\geq 1$,  then $\len(\F) = \prod p_i^{n_i-1}$. 
\end{example}

\begin{remark}
    We can extend our results to all rings by localization to reduce to the Artinian case. For example, over $\Z^n$, we can view the points as in $\Q^n$. 
\end{remark}

Although \cref{thmF} can be applied in the case that each $I_i$ is defined by a product of linear factors, this is not comparable to \cref{th:introP-adic}, since the set of solutions to a linear equation is not necessarily a hyperplane as defined e.g. in \cite{laba2024}.

\section{Combinatorial bounds}\label{sec:combinatorial}

The algebraic proof of \cref{th:fieldUpperBound} does not seem to apply to non-representable matroids.
In this section, we give an elementary combinatorial argument that gives an upper bound on $C(V)$ for matroids, albeit with a quantitatively weaker bound than that given in \cref{th:fieldUpperBound} for representable matroids.

For a set $V$ of dimension vectors, $C_\mathcal{M}(V)$ is defined analogously to $C_A(V)$.
In particular, we say that a matroid $M$ is {\em covered} by a dimension vector $\vec{v}=(v_1,\ldots,v_t)$ if there is a set $\{L_1, \ldots, L_t\}$ of flats with rank $r(L_i) = \dim(L_i)+1 = v_i+1$ such that every element of $M$ is contained in one of the $L_i$.
$M$ is covered by $V$ if it is covered by some $\vec{v} \in V$.
$C_\mathcal{M}(V)$ is the least $N$ such that, if $M$ is a simple matroid on $N$ elements and every restriction of $M$ is covered by $V$, then $M$ is covered by $V$.

\begin{theorem}\label{thm:general weak bound}
     If $V$ is a set of dimension vectors, each with at most $t$ coordinates, and $k = \max_{\vec{v} \in V} \max_{i} v_i$,
    then
    then
    \[C_\mathcal{M}(V) \leq \sum_{i = 0}^{k+1} t^i.\]
\end{theorem}
\begin{proof}
    For consistency with the notation of the rest of the paper, we focus on the ``dimension" instead of the rank of flats, which we define to be one less than the rank.
    
    Suppose that $M$ is a matroid such that every proper restriction of $M$ is covered by $V$, but $M$ is not covered by $V$.
    If $\dim(M) > k+1$, then truncate it to $k+1$ dimensions.
    
    We first show that, if $\Lambda$ is a flat of dimension $j$ with $0 \leq j \leq k+1$, then, for each $x \in \Lambda \cap S$, the set $\Lambda \setminus \{x\}$ is contained in the union of $t$ subspaces, each of dimension at most $j-1$.
    Fix $x \in \Lambda$.
    By assumption, there is $\vec{v} \in V$ and a covering set $L_1, \ldots, L_t$ of $M \setminus x$, with $\dim(L_i) = v_i$ for each $L_i$.
    Since $x$ is not contained in $L_1 \cup \ldots \cup L_t$, it is clear that $\dim(L_i \cap \Lambda) < j$ for each $i \in [t]$, and the claim follows.

    With the claim established in the previous paragraph, a simple inductive argument shows that $|\Lambda| \leq t^j + t^{j-1} + \ldots +1$ for each $j$ dimensional plane $\Lambda$.
    For the base case, if $\dim(\Lambda) = 0$ then $|\Lambda| \leq 1$.
    If $\Lambda$ is a $j$-plane with $j > 0$, then, for each $x \in \Lambda$, the set $\Lambda \setminus \{x\}$ is contained in the union of $t$ planes of dimension at most $j-1$.
    Hence, by the inductive hypothesis, $|\Lambda| \leq t(t^{j-1} + \ldots + 1) + 1$, as claimed.
    The conclusion of the theorem is the case $j=k+1$.
\end{proof}

As in the representable case, one particularly interesting case is when $V=\{k,k,\ldots,k\}$.
As defined in \cite{geelen2015projective}, a matroid is {\em $t+1$-thick} if it is not the union of $t$ hyperplanes.
On particularly important case is that of round matroids; a matroid is {\em round} if and only if it is $3$-thick.
The study of round matroids is closely related to the study of higher matroid connectivity \cite[Chapter 8.6]{oxley2011matroid}.

It follows immediately from \cref{thm:general weak bound} that no rank $n+1$ matroid on at least $1 + \sum_{i=0}^{n}t^i$ elements is minimally $(t+1)$-thick, and from \cref{th:fieldUpperBound} that no representable rank $n+1$ matroid on at least $1+\binom{t+n}{n}$ elements is minimally $(t+1)$-thick.
We conjecture that the representable bound holds in general.

\begin{conjecture}\label{conj:strong general bound}
     If $V$ is a set of dimension vectors, each with at most $t$ coordinates, and $k = \max_{\vec{v} \in V} \max_{i} v_i$,
    then
    \[C_\mathcal{M}(V) \leq \binom{t+k+1}{k+1}.\]
\end{conjecture}

Although the vast majority of matroids are non-representable \cite{knuth1974asymptotic, nelson2018almost}, any extremal example for \cref{conj:strong general bound} should be a highly structured matroid with many small circuits, and such matroids tend to be representable.

It is also interesting to consider the structure of extremal examples for \cref{th:fieldUpperBound}.
The sets described in \cref{ex:triangle} are not the unique extremal examples, for example see \cref{fig:tightExample}.
However, the example shown in \cref{fig:tightExample} can be obtained from that described in \cref{ex:triangle} by adding one circuit on three points (such an operation is called a {\em tightening}).

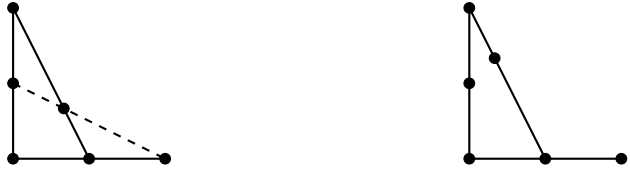
\begin{figure}[h!]
\begin{center}
\begin{tikzpicture}[scale=1]

\filldraw[black] (0,0) circle (2pt);
\filldraw[black] (2,0) circle (2pt);
\filldraw[black] (0,2) circle (2pt);
\filldraw[black] (1,0) circle (2pt);
\filldraw[black] (0,1) circle (2pt);
\filldraw[black] (2/3,2/3) circle (2pt);

\draw[black, thick] (0,0)--(2,0);
\draw[black, thick] (0,0)--(0,2);
\draw[black,thick] (0,2)--(1,0);
\draw[black,thick,dashed] (0,1)--(2,0);

\filldraw[black] (6,0) circle (2pt);
\filldraw[black] (8,0) circle (2pt);
\filldraw[black] (6,2) circle (2pt);
\filldraw[black] (7,0) circle (2pt);
\filldraw[black] (6,1) circle (2pt);
\filldraw[black] (19/3,4/3) circle (2pt);

\draw[black, thick] (6,0)--(8,0);
\draw[black, thick] (6,0)--(6,2);
\draw[black,thick] (6,2)--(7,0);

\end{tikzpicture}
\end{center}
\caption{Two sets of $6$ points with different underlying matroids that are each nearly covered by $2$ lines.
The set of points on the right has the same underlying matroid as that described in \cref{ex:triangle}, and is obtained from the left set by removing the circuit denoted by the dashed line.}\label{fig:tightExample}
\end{figure}

The following precise conjecture was proposed by Rutger Campbell.

\begin{conjecture}\label{conj:structure}
    Every minimal $(t+1)$-thick matroid is a tightening of that described in \cref{ex:triangle}.
\end{conjecture}

It would be interesting to prove \cref{conj:structure} even in the special case of representable matroids.

\section{Nearly covered sets in $(\Z/p^k\Z)^n$}\label{sec:padic-full}


\subsection{Basic $p$-adic geometry}\label{sec:p-adic}


Let $p$ be a prime number, and let $k\in\N$. We define $\Zpk:=\ZZ/p^k\ZZ$, the ring of integers modulo $p^k$, and use $\Zpk^\times$ to denote the multiplicative group of invertible elements of $\Zpk$. We will work in $\Zpk^n$ with coordinates $x=(x_1,x_2,\dots,x_n)$, where $x_j\in\Zpk$ for each $j$. We will also write $\Zpk_\ell =\Z/p^\ell\Z$ for $1\leq \ell\leq k$, so that $\Zpk_k=\Zpk$ and $\Zpk_1=\Z/p\Z$.

We use the notation $|S|$ to denote the cardinality of a set $S$, and the notation $p^j\parallel a$ to mean $p^j\mid a$ but $p^{j+1}\nmid a$. If $x=(x_1,\dots,x_n)\in\Zpk^n$, we will write $p^j\parallel x$ if $p^j|x_i$ for each $i\in\{1,\dots,n\}$ and $p^{j+1}\nmid x_i$ for at least one $i$. The {\em p-adic distance} between two distinct points $x,x'\in \Zpk^n$ is $|x-x'|_p=p^{-\ell}$, where $p^\ell\parallel x-x'$. By convention, we will write $|x-x|_p=p^{-k}$, so that the results below will not require a separate statement in this case.

A {\em direction} in $\Zpk^n$ is an element of the projective space $\PZpk$, defined as follows. Let $\mathbb{S}^{n-1}(\Zpk)$ be the set of all elements of $\Zpk$ that have at least one invertible component, and let
$$
\PP \Zpk^{n-1}= \mathbb{S}^{n-1}(\Zpk)/\Zpk^\times
$$
We will identify a direction $b\in\PZpk$ with a vector $b=(b_1,\dots,b_n)\in\Zpk^n$ such that $b_j\in\Zpk^\times$ for at least one $j$, with the convention that two such vectors $b,b'$ represent the same direction if $b=\lambda b'$ for some $\lambda\in\Zpk^\times$. In particular, when $n=2$, any direction $b\in\PP \Zpk$ may be represented as either $(1,u)$ or $(pu,1)$ with $u\in\Zpk$.

For $b,b'\in \PP \Zpk^{n-1}$, we define the \textit{$p$-adic angle} between $b$ and $b'$ to be $\angle(b,b')=\min_{r\in \Zpk^\times} |b-rb'|_p$. Thus the angle between $b=(1,u)$ and $b'=(pu',1)$ in $\Zpk^2$ is $1$, and the angle between $b=(1,u)$ and $b''=(1,u'')$ is $|u-u''|_p$.

A {\em line} in a direction $b\in \PP \Zpk^{n-1}$ is a set of the form
$$
L_b(a) =\{a + s b:\ s\in \Zpk\} \hbox{ for some }a\in \Zpk^n.
$$
Note that $L_b(a) $ has $|\Zpk|=p^k$ distinct elements. 
The angle between lines $L$ and $L'$, with direction vectors  $b$ and $b'$ respectively, is
$\angle(L,L')=\angle(b,b')$. By a slight abuse of terminology, we will say that the {\it $p$-adic slope} of a line $L_b(a)$ in $\Zpk^2$ is the angle between $b$ and $(0,1)$.

For $0\leq \ell\leq k$, let $\pi_{\ell}:\Zpk^n\rightarrow \Zpk_\ell^n$ be the projection 
\[ \pi_{\ell}(x)=x\bmod{p^\ell}.\]
A \textit{cube on scale $\ell$}, or a {\it $p^{-\ell}$-cube}, is a set of the form 
\[ Q=Q_\ell(x)=\{y\in \Zpk^n:\ |x-y|_p \leq p^{-\ell} \}\subset \Zpk^n \]
for some $x\in \Zpk^2$. Note that a $1$-cube is the entire $\Zpk^n$, and a $p^{-k}$-cube is a single point.

A cube $Q$ on scale $\ell$ may be rescaled to $\Zpk_{k-\ell}^n$ as follows. 
We represent points $x\in Q$ as $x=x'+p^\ell x''$ with $x'\in\{0,1,\dots,p^\ell-1\}^n$ and $x''\in\{0,1,\dots,p^{k-\ell}-1\}^n$. If $x,y$ belong to the same $Q$, then (with the obvious notation) we have $y'=x'$. Therefore the map $\iota_Q:Q\rightarrow R_{k-\ell}^n$ defined by
\[ \iota_Q(x'+p^\ell x'')=x''\]
provides the desired rescaling. 

A distinguishing feature of $p$-adic geometry is that two lines may intersect in more than one point. We describe such intersections in the next two lemmas.

\begin{lemma}
\label{LLangle}  
Let $L,L'\subset \Zpk^n$ be two lines with direction vectors $b,b'$. Assume that $\{a,a'\}\subset L\cap L'$, where $a,a'\in\Zpk^n$ satisfy $|a-a'|_p=p^{-\ell}$ for some $0\leq \ell\leq k-1$. Then $\angle(b,b')\leq p^{k-\ell}$.
\end{lemma}

\begin{proof}
Let $L=L_b(a)$ and $L'=L_{b'}(a)$. Then $a'=a+sb=a+s'b'$ for some $s,s'\in\Zpk$. Since $p^\ell\parallel a-a'=sb=s'b'$ and $p$ does not divide either $b$ or $b'$, we must have $p^\ell\parallel s$ and $p^\ell\parallel s'$. Hence there is an element $r\in \Zpk^\times$ such that $s'=sr$. Let $b''=rb'$, then $b''$ represents the same direction as $b'$, and $a'=a+s'b'=a+sb''$. Hence $sb=sb''$ in $\Zpk$, so that $p^k\mid s(b-b'')$. Since $p^\ell\parallel s$, we must have $p^{k-\ell}|b-b''$, proving the claim.
\end{proof}

\begin{lemma}
\label{Q-properties}  
Let $Q$ be a cube on scale $k-\ell$ for some $1\leq \ell\leq k$. Then:
\begin{enumerate}
\item[(i)] If $L\subset \Zpk^n$ is a line in the direction $b$ intersecting $Q$, then $\iota_Q(Q\cap L)$ is a line in the direction $\pi_{\ell}(b)$ in $\Zpk_{\ell}^n$.
\item[(ii)] Let $L,L'\subset \Zpk^n$ be lines in the directions $b,b'$ respectively. Assume that $\angle(b,b')=p^{-\ell}$, and that the set $L\cap L'\cap Q$ is nonempty. Then $L\cap L'=L\cap Q=L'\cap Q$, and in particular, $|L\cap L'|=p^\ell$.

\end{enumerate}
\end{lemma}

\begin{proof}
For (i), there is nothing to prove when $\ell=k$. Assume now that $\ell<k$, and let $L=L_b(a)$ for some direction $b$ and some $a\in Q$. 
Let $a=a'+p^{\ell}a''$ with $a'\in\{0,1,\dots,p^{\ell}-1\}$. Then 
\[ Q\cap L =\{a+ p^{k-\ell}t b :\ t\in \Zpk \}, \]
so that
\[ \iota_Q(Q\cap L)=\{a''+t \pi_{\ell}(b):\ t\in \Zpk_{\ell}\}\subset \Zpk_{\ell}^n.\]

We now prove (ii). Let $L,L'$ be as in (ii), and let $a\in L\cap L'\cap Q$. Since $\angle(b,b')=p^{-\ell}$, we have $\pi_\ell(b)=\pi_\ell(b')$. By (i), the lines $\iota_Q(Q\cap L)$ and $\iota_Q(Q\cap L')$ are the same in $\Zpk_{\ell}^n$,
hence 
$$L\cap Q=L'\cap Q\subset L\cap L'.$$
For the converse inclusion, suppose we had $a'\in (L\cap L')\setminus Q$. Then $|a-a'|_p> p^{-(k-\ell)}$, and by Lemma \ref{LLangle} we must have $\angle(L,L')<p^{-\ell}$, contradicting the assumptions of (ii).
\end{proof}

\begin{lemma}\label{lemma-ratio}
Let $n=2$. Suppose that the line $L=L_b(a)$ passes through a point $a'$ such that $|a_1-a'_1|_p=p^{-j}$ and $|a_2-a'_2|_p=p^{-\ell}$, where $\ell>j$. Then $L$ makes angle at most $p^{-\ell+j}$ with $(1,0)$.
\end{lemma}

\begin{proof}
We may assume that $b=(b_1,b_2)$ with one of $b_1,b_2$ equal to 1. We have $a'=a+tb$ for some $t\in\Zpk$, so that
$$
a'_1-a_1=tb_1,\ \ a'_2-a_2=tb_2.
$$
If we had $b_2=1$, it would follow that $p^\ell|(a'_2-a_2)=t$. But then $p^\ell$ would also divide $tb_1=a'_1-a_1$, a contradiction. Therefore $b_1=1$. It follows that $p^j\parallel t$, so that $p^{\ell-j}\mid b_2$, proving the lemma.
\end{proof}


\subsection{Large nearly covered sets in $(\Z/p^k\Z)^2$}\label{sec:fullExample}

We continue to use the notation of Section \ref{sec:p-adic}, with $n=2$. We also let $t\in\N$.
When $k\geq 2$ and $p$ is sufficiently large relative to $t$, Example \ref{ex:triangle} can be improved by taking advantage of the multiple scales available in $\Zpk$. 
Our result is as follows.

\begin{theorem}\label{th:p-adicConstructionMain}
Let $2\leq t < \frac{\sqrt{p}}{4}$ and $k\geq 2$. 
Define the parameters $\ell$ and $M$ as follows:
\begin{itemize}
    \item If $k=2$, let $M=\ell=1$.
    \item If $k\geq 3$, let $\ell=\lfloor \log_p k\rfloor+2$ and
    $M=\lfloor (k-1)/\ell\rfloor$.   
\end{itemize}
Let also $t'=t+M-1$.
Then there exists a set $S$ in $(\Z/p^k\Z)^2$ of size at least 
    \begin{equation}\label{S-size-padic}
    |S|=  2^M \binom{t+1}{2}+2^M-1
\end{equation}
such that $S$ cannot be covered by $t'$ lines but $S\setminus \{x\}$ for any $x\in S$ can be covered by $t'$ lines. 

\end{theorem}

The set constructed in Theorem \ref{th:p-adicConstructionMain} has cardinality strictly larger than $\binom{t'+2}{2}$ (the cardinality of the set in Example (\ref{ex:triangle}) with $t$ replaced by $t'$) for all $k\geq 2$ and $t\geq 2$. Figures \ref{flat-triangle-fig1} and \ref{flat-triangle-fig2} show this for $k=2$.

Furthermore, assume that $p>k$. Then $\ell=2$ and
$M=\lfloor \frac{k-1}{2}\rfloor$, so that
$$
|S|=  2^{\lfloor \frac{k-1}{2}\rfloor} 
\left( \binom{t+1}{2}+1\right) -1
$$
as claimed in Theorem \ref{th:introP-adic}.

The main idea of our construction is illustrated in Figures \ref{flat-triangle-fig1} and \ref{flat-triangle-fig2}. We start with the triangle from Example \ref{ex:triangle} and flatten it so that any line passing through two distinct points of the triangle makes a low angle with the horizontal line. We then add a second copy of the flat triangle, translated by a small increment so that a low-slope line passing through a point of the triangle must also pass through its companion point. Finally, we add one more point that is not colinear with any two of the triangle points. To cover the entire set, we need to cover the triangle and add one more line for the extra point; however, if any point is removed from the set, one line can be dropped as shown in Figure \ref{flat-triangle-fig2}.

\begin{figure}[h!]
\begin{center}
\begin{tikzpicture}[scale=.5]

\filldraw[black] (0,2) circle (4pt);
\filldraw[black] (0.5,2) circle (4pt);
\filldraw[black] (3,1) circle (4pt);
\filldraw[black] (3.5,1) circle (4pt);
\filldraw[black] (6,0) circle (4pt);
\filldraw[black] (6.5,0) circle (4pt);
\filldraw[black] (9,1) circle (4pt);
\filldraw[black] (9.5,1) circle (4pt);
\filldraw[black] (12,0) circle (4pt);
\filldraw[black] (12.5,0) circle (4pt);
\filldraw[black] (18,0) circle (4pt);
\filldraw[black] (18.5,0) circle (4pt);
\filldraw[black] (10,10) circle (4pt);


\draw[black, thick] (0.5,2)--(3,1);
\draw[black, thick] (0.5,2)--(9,1);
\draw[black, thick] (3.5,1)--(6,0);
\draw[black, thick] (6.5,0)--(12,0);
\draw[black, thick] (9.5,1)--(18,0);
\draw[black, thick] (12.5,0)--(18,0);
\draw[black, thick] (6,10)--(14,10);

%
%

\end{tikzpicture}
\end{center}
\caption{The flattened triangle with two points at each vertex, plus an additional point that requires an extra line.}
\label{flat-triangle-fig1}
\end{figure}
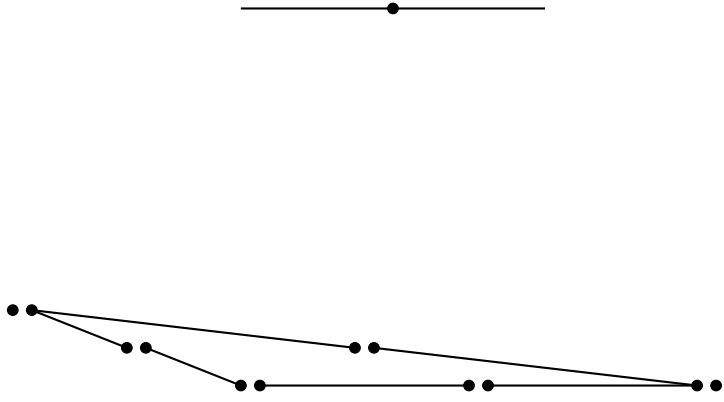

\begin{figure}[h!]
\begin{center}
\begin{tikzpicture}[scale=.5]

\filldraw[black] (0,2) circle (4pt);
\filldraw[black] (0.5,2) circle (4pt);
\filldraw[black] (3,1) circle (4pt);
\filldraw[black] (3.5,1) circle (4pt);
\filldraw[black] (6,0) circle (4pt);
\filldraw[black] (6.5,0) circle (4pt);
\draw[black, thick] (9,1) circle (4pt);
\filldraw[black] (9.5,1) circle (4pt);
\filldraw[black] (12,0) circle (4pt);
\filldraw[black] (12.5,0) circle (4pt);
\filldraw[black] (18,0) circle (4pt);
\filldraw[black] (18.5,0) circle (4pt);
\filldraw[black] (10,10) circle (4pt);

\draw[red, thick] (10,10)--(9.5,1);

\draw[black, thick] (0.5,2)--(3,1);
\draw[black, thick, dashed] (0.5,2)--(8.9,1);
\draw[black, thick] (3.5,1)--(6,0);
\draw[black, thick] (6.5,0)--(12,0);
\draw[black, thick, dashed] (9.5,1)--(18,0);
\draw[black, thick] (12.5,0)--(18,0);

%
%

\end{tikzpicture}
\end{center}
\caption{If we remove one of the points in the triangle, then its companion and the extra point above can be covered by one line, and then one of the lines covering the triangle is no longer needed.}
\label{flat-triangle-fig2}
\end{figure}
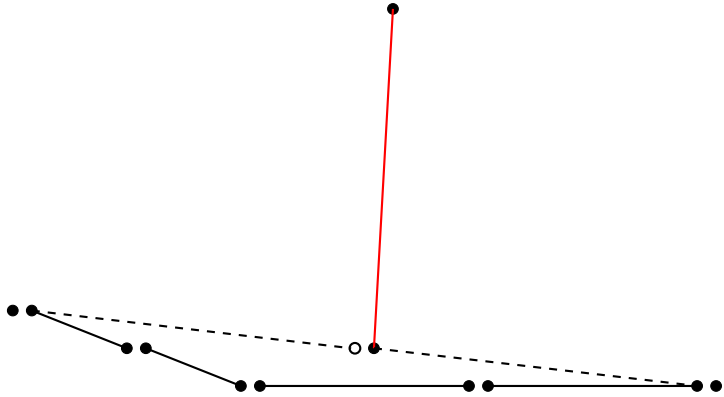

We now proceed with the rigorous proof. For the simple example described in Figures \ref{flat-triangle-fig1} and \ref{flat-triangle-fig2}, we recommend reading the proof below with $k=2$ and $M=\ell=1$. The general construction giving the bound in Theorem \ref{th:p-adicConstructionMain} is based on iterating the argument.

Let $t\in\N$ with $t\geq 2$. Let $p,r$ be primes such that
\begin{equation}\label{size-pr}
2\leq t < \frac{\sqrt{p}}{4}<r<\frac{\sqrt{p}}{2}.
\end{equation}
We note that the first two inequalities imply that $p>64$. The existence of a prime $r$ satisfying (\ref{size-pr}) is then guaranteed by Bertrand's Postulate.

\begin{lemma}\label{a-distances}
Let $k\geq 3$. For $m,m'\in \{1,\hdots,k\}$ such that $m\neq m'$, we have $|mp-m'p|_p\geq p^{-\ell+1}$.
\end{lemma}

\begin{proof}
For $1\leq m\leq k$, we have 
$\log_p m \leq \log_p k < \ell -1$,
so that $m<p^{\ell -1}$ and $mp<p^\ell$. Thus $p^{\ell-1}$ is the highest power of $p$ that may divide $mp-m'p$, as claimed.
\end{proof}

Let 
$$T=\{x=(x_1,x_2)\in \Zpk^2:\ 0\leq x_1+x_2\leq t+1, \ x_1\geq 1.\ ,x_2\ge 1\}.$$
For two sets $S,S'\subset\Zpk^2$, we write $S+S'=\{x+y:\ x\in S,y\in S'\}$.
If $S'=\{y\}$ is a singleton, we write $S+y=\{x+y:\ x\in S\}$.

\begin{lemma}\label{lem-multiscale-flat-triangle}
Consider the linear mapping $F:\Zpk^2\to \Zpk^2$ defined by
$$
F(z_1,z_2)=(z_1+rz_2, p^{k-1}z_2).
$$
\begin{itemize}
\item[(i)] 
We have $F(T)\subset \{x=(x_1,p^{k-1} x_2)\in\Zpk^2:\ x_1,x_2\in\{1,2,\dots,p-1\}\}$. Furthermore, if $z,w\in T$ are distinct, then $x=F(z)$ and $y=F(w)$ satisfy $x_1\neq y_1$. In particular, $F$ is injective on $T$.

\item[(ii)] Define $\Xi_0:=\{(0,0)\}$, and 
\begin{equation}\label{def-xi}
\Xi_m =\left\{ \left( \sum_{i=1}^m \nu_i p^{i\ell},0\right):\ \nu_i\in\{0,1\}\right\}
\end{equation}
for $1\leq m\leq M$. Then for any two distinct points $x,y\in F(T)+\Xi_m$, any line passing through both $x$ and $y$ has $p$-adic slope at most $p^{-k+m\ell}$. 

\item[(iii)] The set $F(T)$ can be covered by $t$ lines but not by $t-1$ lines. Moreover, for any $x\in T$, the set $T\setminus\{x\}$ can be covered by $t-1$ lines with $p$-adic slope at most $p^{-k+1}$.

\end{itemize}
\end{lemma}

\begin{proof}
We start with (i). Let $x=F(z)$ for some $z\in T$. By (\ref{size-pr}), we have 
$$
1\leq z_1+rz_2< \frac{\sqrt{p}}{4}+  \frac{\sqrt{p}}{2} \cdot \frac{\sqrt{p}}{4}<p.
$$
Moreover, let $y=F(w)$ for some $w\in T$. If $y_1=x_1$, then $z_1+rz_2=w_1+rw_2$, so that $z_1-w_1=r(w_2-z_2)$ is divisible by $r$. Since $z_1,w_1\in\{1,\dots, t \}$ with $t<r$, we must have $z_1=w_1$. But then $rz_2=rw_2$ implies that $z=w$.

Next, we prove (ii). Let $x=x'+\xi$ and $y=y'+\eta$ with $x',y'\in F(T)$ and $\xi,\eta\in\Xi_m$. If $x'_1\neq y'_1$, then 
$$|x_1-y_1|_p = |x'_1-y'_1|_p = 1\hbox{ and } 
|x_2-y_2 |_p =  |x'_2-y'_2|_p\leq p^{-k+1}.
$$
By Lemma \ref{lemma-ratio}, any line through $x$ and $y$ has slope at most $p^{-k+1}$, proving (ii) in this case.
If on the other hand $x'_1=y'_1$, by (i) we must have $x'=y'$. Since $x\neq y$, we have $m\geq 1$ and $\xi\neq\eta$, so that
$$
|x_1-y_1|_p = |\xi_1-\eta_1|_p  \geq p^{-m\ell} \hbox{ and } 
x_2=y_2.
$$
By Lemma \ref{lemma-ratio} again, any line through $x$ and $y$ has slope at most $p^{-k+m\ell}$, as claimed.

We now prove (iii).  By Example \ref{ex:triangle}, $T$ can be covered by $t$ lines but not by $t-1$ lines, but for any $x\in T$, the set $T\setminus\{x\}$ can be covered by $t-1$ lines. 

Clearly, $F(T)$ can be covered by the $t$ lines $\{x:\ x_2=j\}$ for $j=1,2,\dots,t$. We now prove that for any $x\in F(T)$, the set $F(T)\setminus\{x\}$ can be covered by $t-1$ lines. Let $x=F(z)$ for $z\in T$, and let $L_1,\dots,L_{t-1}$ be lines covering $T\setminus\{x\}$. Then the sets $F(L_1),\dots,F(L_{t-1})$ cover $F(T)\setminus\{x\}$. It remains to show that if $L$ is a line, then $F(L)$ can be covered by a line. Let $L=\{a+sb:\ s\in\Zpk\}$. Then
$$
F(L)=\{F(a)+sF(b):\ s\in\Zpk\}.
$$
This is a line if $F(b)=(b_1+rb_2,p^{k-1}b_2)$ is a direction. Suppose therefore that this is not the case, so that  $b_1+rb_2=p^j u$ for some $u\in\Zpk^\times$ and $1\leq j\leq k$. If $j=k$, then $F(b)=(0,p^{k-1}b_2)$, and $F(L)\subset L_{(0,1)}(a)$. If $1\leq j\leq k-1$, then $v:=p^{-j}F(b)=(u,p^{k-1-j}b_2)$ is a direction, and $F(L)\subset L_{v}(a)$.

Let $L$ be one of the lines covering $F(T)\setminus\{x\}$. 
If $L$ contains two distinct points of $F(T)$, it follows from (ii) that $L$ has slope at most $p^{-k+1}$. If on the other hand 
$L\cap F(T)=\{y\}$ for some $y\in F(T)$, we may simply replace $L$ by $L_{(1,0)}(y)$. This proves the statement about slopes.

To complete the proof of (iii), we need to show that $F(T)$ cannot be covered by $t-1$ lines. Assume towards contradiction that such a covering exists, and let $L$ be one of the covering lines. We may assume that $L=L_b(x)$ for some $x=F(z)$, where $z\in T$. As shown above, we may further assume that $b=(1,p^{k-1}u)$ for some $u\in \{0,1,\dots,p-1\}$. Let
\begin{equation}\label{padic-eL'}
L'=L_{v}(z), \hbox{ where }v=(1-ru,u).
\end{equation}
It suffices to prove the following claim. Let $w\in T$, $w\neq z$, and let $y=F(w)$. If $y\in L$, then $w\in L'$. Indeed, if we can prove this, then $T$ is covered by the lines $L'$ corresponding via (\ref{padic-eL'}) to the lines $L$ covering $F(T)$. But by Example \ref{ex:triangle}, a covering of $T$ requires at least $t$ lines.

We now prove the claim. Assume that $y\in L$ as above, so that $y=x+sb=x+(s,p^{k-1}su)$ for some $s\in\Zpk$. Then 
$$
y_1-x_1=s,\ y_2-x_2=p^{k-1}su.
$$
But we also have $x=F(z)=(z_1+rz_2,p^{k-1}z_2)$ and $y=F(w)=(w_1+rw_2,p^{k-1}w_2)$, so that
$$
w_2-z_2=p^{-(k-1)} (y_2-x_2) = su,
$$
$$
w_1-z_1= (y_1-x_1) - r (w_2 - z_2) = s - sru.
$$
Hence $w-z=sv$, proving the claim. This completes the proof of the lemma.
\end{proof}

We are now ready to construct our example. 
Let $K_0:= F(T)$. 
For $m\in \{1,\hdots,M\}$, let 
$${\bf a}_m: =  (mp,p^{k-m\ell-1}),$$
\begin{equation}\label{K-inductive}
K'_m := K_{m-1}\cup (K_{m-1}+(p^{m\ell},0)), \ \ 
K_m: =K'_m\cup\{{\bf a}_m\}.
\end{equation}
Equivalently, we have
\begin{equation}\label{K-tower}
K_m=(K_0+\Xi_m)\cup\bigcup_{j=1}^m ({\bf a}_j+p^{j\ell}\, \Xi_{m-j}),
\end{equation}
where $\Xi_j$ were defined in (\ref{def-xi}). 
We note that
$$
|K_m|=2^m\binom{t+1}{2} + \sum_{j=1}^m 2^{m-j} = 2^m\binom{t+1}{2} + 2^m -1.
$$
We claim that the set $S:=K_M$ satisfies the conclusions of Theorem \ref{th:p-adicConstructionMain}. Indeed, Equation (\ref{S-size-padic}) follows from the above with $m=M$. It remains to prove that $K_M$ has the desired properties with regard to being covered by lines. We now prove this by induction in $m$.

\begin{lemma}\label{lem-iter-K-angle}
For distinct $x,y\in K'_m$ with $m=1,\hdots,M $, any line $L$ joining $u$ and $v$ has $p$-adic slope at most $p^{-k+m\ell}$.
\end{lemma}

\begin{proof}
By (\ref{K-tower}), we have
$$
K'_m= 
(K_0+\Xi_m)\cup\bigcup_{j=1}^{m-1} ({\bf a}_j+p^{j\ell}\, \Xi_{m-j}).
$$
We consider the following cases.
\begin{itemize}
\item Suppose that either $m=1$, or else $m\geq 2$ and $x,y\in K_0+\Xi_m$. Then the conclusion follows from Lemma \ref{lem-multiscale-flat-triangle} (ii).

\item Let $x\in K_0+\Xi_m$ and $y\in {\bf a}_j+p^{j\ell} \,\Xi_{m-j}$ for some $1\leq j\leq m-1$.
Then $|x_1-y_1|_p=1$ and $|x_2-y_2|_p=p^{-k+j\ell+1}\leq p^{-k+m\ell}$, so that the claim follows from Lemma \ref{lemma-ratio}.

\item Let $x,y\in {\bf a}_j+p^{j\ell}\, \Xi_{m-j}$ for some $1\leq j\leq m-1$.
Then $|x_1-y_1|_p\geq p^{-(m\ell)}$ and $x_2=y_2$. By Lemma \ref{lemma-ratio},
$L$ has slope at most $p^{-k+m\ell}$.

\item Let $x\in {\bf a}_i+p^{i\ell}\,\Xi_{m-i}$ and $y\in {\bf a}_j+p^{j\ell}\, \Xi_{m-j}$ for some $1\leq i<j\leq m-1$. This can happen only when $m\geq 3$, so that $k\geq 3$. By Lemma \ref{a-distances}, we have $|x_1-y_1|_p\geq p^{-\ell+1}$. Since $|x_2-y_2|_p=p^{-k+j\ell +1}$, it follows by Lemma \ref{lemma-ratio} that
$L$ has slope at most $(p^{-k+j\ell+1})/(p^{-\ell+1})=p^{-k+(j+1)\ell}\leq p^{-k+m\ell}$.
\end{itemize}
\end{proof}

\begin{proposition}
For $m=0,1,\hdots,M$, the set $K_m$ cannot be covered by $t+m-1$ lines. However, for any $x\in K_m$, the set $K_m\setminus\{x\}$ can be covered by $t+m-1$ lines with $p$-adic slope at most $p^{-k+m\ell}$ if $m\geq 1$, and at most $p^{-k+1}$ if $m=0$.
\end{proposition}

\begin{proof}
We proceed by induction in $m$. For the base case $m=0$, the conclusion follows from Lemma~\ref{lem-multiscale-flat-triangle} (iii). Assume now that $m\geq 1$, and that the proposition has been proved with $m$ replaced by $m-1$. We will prove it for $m$.

We first prove that $K_{m}$ cannot be covered by $t+m-1$ lines. Assume towards contradiction that $K_{m}$ is covered by the lines $L_1,\dots,L_{t+m-1}$, and that $\mathbf{a}_m\in L_{1}$. Suppose first that $L_{1}$ contains two distinct points $x,y\in K'_m$. By Lemma \ref{lem-iter-K-angle}, $L_1$ must have $p$-adic slope at most $p^{-k+m\ell}$. Since we also have $K_{m-1}\subset \Zpk\times p^{k-m\ell}\Zpk$, it follows that $L_{1}\subset\Zpk\times p^{k-m\ell}\Zpk$. In particular, $L_{1}$ cannot contain ${\bf a}_m=(pm,p^{k-m\ell -1})$, contradicting our assumption.

Therefore $L_1$ contains at most one point of $K'_m$.
It follows that at least one of the sets $K_{m-1}$ and $(K_{m-1}+(p^{m\ell},0))$ is covered by the remaining lines $L_2,\dots,L_{t+m-1}$. This contradicts the inductive hypothesis for $m-1$. Hence $K_{m}$ cannot be covered by $t+m-1$ lines.

We now prove that $K_{m}\setminus\{x\}$ can be covered by $t+m-1$ lines for any $x\in K_{m}$. If $x={\bf a}_m$, then $K_{m}\setminus \{x\}=K'_m$ can be covered by the $t+m-1$ lines $L_{(1,0)}((0,jp^{k-1}))$ for $j=1,\dots,t$ and $L_{(1,0)}((0,p^{k-j\ell+1}))$ for $j=1,\dots,m-1$. Assume now that $x\in K'_m$. Then $x$ is one of the points $y$, $y+(p^{m\ell},0)$ for some $y\in K_{m-1}$. Let $L_{t+m-1}$ be a line through ${\bf a}_m$ and the remaining one of these two points. 
Now, without loss of generality we assume $x=y\in K_{m-1}$. We are left with the set
$$
(K_{m-1}\setminus\{x\}) \cup ((K_{m-1}\setminus\{x\})+(p^{m\ell},0)).
$$
By the inductive hypothesis, $K_{m-1}\setminus\{x\}$ can be covered by $t+m-2$ lines $L_1,\dots,L_{t+m-2}$, all with slopes at most $p^{-k+m\ell}$. 
We claim that whenever one of these lines passes through a point $z\in K_{m-1}$, it also passes through $z'=z+(p^{m\ell},0)$. Indeed, we may write the line as $L_b(z)$ with $b=(1,p^{k-m\ell}u)$ for some $u\in\Zpk$. Then $z-z'=(p^{m\ell},0)=p^{m\ell}b$, and $z'\in L_b(z)$ as claimed. 
Hence, $L_1,\hdots,L_{t+m-2}$ cover both $K_{m-1}$ and $K_{m-1}+(p^{m\ell},0))$. 
This concludes the proof of the proposition.
\end{proof}

\subsection{Upper bound for $p$-adic lines}\label{sec:p-adicBound}

We say that a set $S$ of points in $\Zpk^n$ is {\em nearly covered} by $t$ lines if each proper subset of $S$ is contained in the union of some set of $t$ lines, but no set of $t$ lines contains $S$ itself.

\begin{theorem}
    If $S \subseteq \mathcal{R}^2$ is nearly covered by $t$ lines, then $|S| \leq t \left(1 + k^{-1}t\right)^k+1$ if $t \geq k$, and $|S| \leq t2^t+1$ if $k \geq t$.
\end{theorem}
\begin{proof}
    For $0 \leq \ell \leq k$, denote by $f(\ell,t)$ the largest number of points in any set $T$ such that
    \begin{enumerate}
        \item $T$ is contained in the intersection of a line $L$ and a $p^{-\ell}$-cube $Q$, and
        \item for each point $P \in T$, there is a set $\mathcal{L}_P$ of at most $t$ lines such that $T\setminus \{P\} \subset \bigcup_{L' \in \mathcal{L}_P}L'$ and $P \notin \bigcup_{L' \in \mathcal{L}_P}L'$.
    \end{enumerate}
    Clearly, $f(k,t)=1$, since each cube on scale $k$ contains a single point.
    We claim that,
    \begin{equation}\label{eq:qAdicUBRecursion}
    f(\ell,t) \leq \max_{0 \leq j \leq t} (j+1) f(\ell+1, t-j)\end{equation}
    for $0 \leq \ell \leq k-1$.

    Indeed, let $T$, $L$, and $Q$ be as above.
    Suppose that $T$ has nonempty intersection with exactly $j+1$ distinct $p^{-\ell-1}$-cubes $Q_1,\dots,Q_{j+1}$ contained in $Q$.
    Let $Q'=Q_{j_0}$ be one of them, let $P \in T \cap Q'$, and denote $\mathcal{L}=\mathcal{L}_P$.

    Let $L'\in\mathcal{L}$.
    By \cref{Q-properties} part (ii), there is a cube $\tilde{Q}$ on some scale $\tilde{\ell}$ such that $L\cap L'=L\cap \tilde{Q}=L'\cap\tilde{Q}$. We have $L'\cap L\cap Q\neq\emptyset$; on the other hand, $P\in L\cap Q$ and $P\not\in L'\cap Q$, so that $L'\cap L\cap Q'\subsetneq L\cap Q$. It follows that $\tilde{\ell}\geq\ell+1$, and that $\tilde{Q}$ is contained in one of the cubes $Q_j$.

We can thus partition $\mathcal{L}$ into $\mathcal{L}_1 = \{L' \in \mathcal{L}: L' \cap L \subseteq Q'\}$ and $\mathcal{L}_2 = \{L' \in \mathcal{L}: L \cap L' \cap Q' = \emptyset\}$.
    Furthermore, if $L' \in \mathcal{L}_2$, then $L' \cap L \subseteq Q_j$ for some $j\neq j_0$. Since for each $j\neq j_0$ there must be at least one such line, 
    we have that $|\mathcal{L}_2| \geq j$, and so $|\mathcal{L}_1| \leq t-j$.
    Since the choice of $P$ was arbitrary, this implies that $|T \cap Q'| \leq f(\ell + 1, t-j)$, and \cref{eq:qAdicUBRecursion} follows directly.

    From \cref{eq:qAdicUBRecursion}, we see that $f(0,t) \leq \prod_{0 \leq i \leq k-1} (j_i+1)$ for some set of integers $j_i$ such that each $j_i \geq 0$ and $\sum j_i = t$.
    Maximizing this function, we see that $f(0,t) \leq (1+k^{-1}t)^k$ if $t \geq k$, and $f(0,t) \leq 2^t$ if $k \geq t$.
    Since each proper subset of $S$ is contained in the union of $t$ sets that satisfy the above hypotheses for $T$ with $\ell = 0$, the conclusion of the theorem follows immediately.
\end{proof}

\section{Acknowledgments}

Izabella {\L}aba was supported by NSERC Discovery Grant 22R80520. Ben Lund was supported by the Institute for Basic Science (IBS-R029-C1).

Part of the research by Hailong Dao and Ben Lund was carried out at Vietnam Institute for Advanced Study in Mathematics (VIASM), and we thank VIASM for their hospitatilty.
Part of the research by all authors was conducted at the IBS-DIMAG workshop on combinatorics and geometric measure theory, and we thank the Institute for Basic Science for their hospitality.

Ben Lund thanks Boris Bukh, Alexander Clifton, Rutger Campbell, and Peter Nelson for helpful conversations.

\bibliographystyle{plain}
\bibliography{nearlyCovered}

\end{document}